\documentclass[11pt]{amsart}
\usepackage{fullpage,url,amssymb,mathrsfs,color,multicol,enumerate,moreverb}
\usepackage[all]{xy} % for commutative diagrams
\definecolor{backgroundcolor}{rgb}{1,1,0.9}
\numberwithin{equation}{section}

\usepackage{color}

\newcommand{\blank}[1]{{}}

\def\bbar#1{\setbox0=\hbox{$#1$}\dimen0=.2\ht0 \kern\dimen0 }
\newcommand{\defi}[1]{\textsf{#1}} % for defined terms

\newenvironment{romanenum}{\hfill \begin{enumerate} }{\end{enumerate}}
\newenvironment{alphenum}{\hfill \begin{enumerate} }{\end{enumerate}}

  \newcommand{\FF}{{\mathbb F}}

 \newcommand{\QQ}{{\mathbb Q}}

\newcommand{\ZZ}{{\mathbb Z}}

\def\bbar#1{\setbox0=\hbox{$#1$}\dimen0=.2\ht0 \kern\dimen0 \overline{\kern-\dimen0 #1}}
\newcommand{\Qbar}{{\overline{\mathbb Q}}}

 \newcommand{\FFbar}{\overline{\FF}}

\newcommand{\calC}{{\mathcal C}}

\newcommand{\calI}{{\mathcal I}}

\newcommand{\calK}{{\mathcal K}}

\newcommand{\calM}{{\mathcal M}}
\newcommand{\calN}{{\mathcal N}}

\newcommand{\OO}{{\mathcal O}}

\DeclareMathOperator{\tr}{tr}

\DeclareMathOperator{\Frob}{Frob}

\DeclareMathOperator{\Aut}{Aut}
\DeclareMathOperator{\Gal}{Gal}

\newcommand{\proj}{{\operatorname{proj}}}

\newcommand{\ab}{{\operatorname{ab}}}

\newcommand{\GL}{\operatorname{GL}}
\newcommand{\SL}{\operatorname{SL}}
\newcommand{\PGL}{\operatorname{PGL}}
\newcommand{\PSL}{\operatorname{PSL}}

\def\CC{\mathbb C}

\newtheorem{theorem}{Theorem}[section]
\newtheorem{lemma}[theorem]{Lemma}

\theoremstyle{definition}

\theoremstyle{remark}
\newtheorem{remark}[theorem]{Remark}

\definecolor{webcolor}{rgb}{0,0,1}
\definecolor{webbrown}{rgb}{.6,0,0}
\usepackage[
        colorlinks,
        linkcolor=webbrown,  filecolor=webbrown,  citecolor=webbrown, 
        backref,
        pdfauthor={David Zywina}, 
]{hyperref}
\usepackage[alphabetic,backrefs,lite]{amsrefs} % for bibliography

\begin{document}

\title[Modular forms and some cases of the inverse Galois problem]{Modular forms and some cases of the inverse Galois problem}
%\keywords{}
%12F12 Inverse Galois theory
%11F11 Holomorphic modular forms of integral weight
\author{David Zywina}
\address{Department of Mathematics, Cornell University, Ithaca, NY 14853, USA}
\email{zywina@math.cornell.edu}
\urladdr{http://www.math.cornell.edu/~zywina}

\begin{abstract} 
We prove new cases of the inverse Galois problem by considering the residual Galois representations arising from a fixed newform.   Specific choices of weight $3$ newforms will show that there are Galois extensions of $\QQ$ with Galois group $\PSL_2(\FF_p)$ for all primes $p$ and $\PSL_2(\FF_{p^3})$ for all odd primes $p \equiv  \pm 2, \pm 3, \pm 4, \pm 6 \pmod{13}$.
\end{abstract}

\maketitle

\section{Introduction}

The \defi{Inverse Galois Problem} asks whether every finite group is isomorphic to the Galois group of some extension of $\QQ$.   There has been much work on using modular forms to realize explicit simple groups of the form $\PSL_2(\FF_{\ell^r})$ as Galois groups of extensions of $\QQ$, cf.~\cite{MR0419358}, \cite{MR1352266}, \cite{MR1800679},  \cite{MR1879665}, \cite{MR2512358}.   For example, \cite{MR1800679}*{\S3.2} shows that $\PSL_2(\FF_{\ell^2})$ occurs as a Galois group of an extension of $\QQ$ for all primes $\ell$ in a explicit set of density $1-1/2^{10}$ (and for primes $\ell \leq 5000000$).    Also it is shown in \cite{MR1800679} that $\PSL_2(\FF_{\ell^4})$ occurs as a Galois group of an extension of $\QQ$ when  $\ell \equiv 2,3 \pmod{5}$ or $\ell\equiv \pm 3,\pm 5,\pm 6,\pm 7 \pmod{17}$.

The goal of this paper is to try to realize more groups of the form $\PSL_2(\FF_{\ell^r})$ for \emph{odd} $r$.  We will achieve this by working with newforms of odd weight;  the papers mentioned above focus on even weight modular forms (usual weight $2$).    We will give background and describe the general situation in \S\ref{SS:general result}.   In \S\ref{SS:weight 3 level 27} and \S\ref{SS:weight 3 level 160}, we will use specific newforms of weight $3$ to realize many groups of the form $\PSL_2(\FF_{\ell^r})$ with $r$ equal to $1$ and $3$, respectively.

Throughout the paper, we fix an algebraic closure $\Qbar$ of $\QQ$ and define the group $G:= \Gal(\Qbar/\QQ)$.  For a ring $R$, we let $\PSL_2(R)$ and $\PGL_2(R)$ be the quotient of $\SL_2(R)$ and $\GL_2(R)$, respectively, by its subgroup of scalar matrices (in particular, this notation may disagree with the $R$-points of the corresponding group scheme $\PSL_2$ or $\PGL_2$).      

\subsection{General results} \label{SS:general result}

Fix a non-CM newform $f(\tau)=\sum_{n=1}^\infty a_n q^n$ of weight $k>1$ on $\Gamma_1(N)$, where the $a_n$ are complex numbers and $q=e^{2\pi i \tau}$ with $\tau$ a variable of the complex upper-half plane.    Let $\varepsilon\colon (\ZZ/N\ZZ)^\times \to \CC^\times$ be the nebentypus of $f$.

Let $E$ be the subfield of $\CC$ generated by the coefficients $a_n$; it is also generated by the coefficients $a_p$ with primes $p\nmid N$.   The field $E$ is a number field and all the $a_n$ are known to lie in its ring of integers  $\OO$.  The image of $\varepsilon$ lies in $E^\times$.    Let $K$ be the subfield of $E$ generated by the algebraic integers $r_p:=a_p^2/\varepsilon(p)$ for primes $p\nmid N$; denote its ring of integer by $R$.

Take any non-zero prime ideal $\Lambda$ of $\OO$ and denote by $\ell=\ell(\Lambda)$ the rational prime lying under $\Lambda$.  Let $E_\Lambda$ and $\OO_\Lambda$ be the completions of $E$ and $\OO$, respectively, at $\Lambda$.   From Deligne \cite{Deligne71-179}, we know that there is a continuous representation 
\[
\rho_\Lambda \colon G  \to \GL_2(\OO_\Lambda)
\]
such that for each prime $p\nmid N\ell$, the representation $\rho_\Lambda$ is unramified at $p$ and satisfies
\begin{equation} \label{E:trace and det}
\tr(\rho_\Lambda(\Frob_p)) = a_p  \quad \text{ and } \quad \det(\rho_\Lambda(\Frob_p))= \varepsilon(p) p^{k-1}.
\end{equation}
The representation $\rho_\Lambda$ is uniquely determined by the conditions (\ref{E:trace and det}) up to conjugation by an element of $\GL_2(E_\Lambda)$.    By composing $\rho_\Lambda$ with the natural projection arising from the reduction map $\OO_\Lambda \to \FF_\Lambda:=\OO/\Lambda$, we obtain a representation
\[
\bbar\rho_\Lambda \colon G \to \GL_2(\FF_\Lambda).
\]
Composing $\bbar\rho_\Lambda$ with the natural quotient map $\GL_2(\FF_\Lambda)\to \PGL_2(\FF_\Lambda)$, we obtain a homomorphism
\[
\bbar\rho_\Lambda^\proj \colon G \to \PGL_2(\FF_\Lambda).
\]
Define the field $\FF_\lambda:=R/\lambda$, where $\lambda:= \Lambda\cap R$.   There are natural injective homomorphisms $\PSL_2(\FF_\lambda) \hookrightarrow \PGL_2(\FF_\lambda) \hookrightarrow \PGL_2(\FF_\Lambda)$ and $\PSL_2(\FF_\Lambda) \hookrightarrow \PGL_2(\FF_\Lambda)$ that we shall view as inclusions.

The main task of this paper is to describe the group $\bbar\rho_\Lambda^\proj(G)$ for all $\Lambda$ outside of some \emph{explicit} set.     The following theorem of Ribet gives two possibilities for  $\bbar\rho_\Lambda^\proj(G)$ for all but finitely many $\Lambda$.

\begin{theorem}[Ribet] \label{T:Ribet}
There is a finite set $S$ of non-zero prime ideals of $R$ such that if $\Lambda$ is a non-zero prime ideal of $\OO$ with $\lambda:=R\cap \Lambda  \notin S$, then the group $\bbar\rho_\Lambda^\proj(G)$ is conjugate in $\PGL_2(\FF_\Lambda)$ to either $\PSL_2(\FF_\lambda)$ or $\PGL_2(\FF_\lambda)$.
\end{theorem}
\begin{proof}
As noted in \S3 of \cite{MR2806684}, this is an easy consequence of \cite{MR819838}.   
\end{proof}

We will give a proof of Theorem~\ref{T:Ribet} in \S\ref{S:effective Ribet} that allows one to compute such a set $S$.    There are several related results in the literature; for example, Billerey and Dieulefait \cite{MR3188630} give an version of Theorem~\ref{T:Ribet} when the nebentypus $\varepsilon$ is trivial.

 We now explain how to distinguish the  two possibilities from Theorem~\ref{T:Ribet}.  Let $L\subseteq \CC$ be the extension of $K$ generated by the square roots of the values $r_p=a_p^2/\varepsilon(p)$ with $p\nmid N$; it is a finite extension of $K$ (moreover, it is contained in a finite cyclotomic extension of $E$).  

\begin{theorem}  \label{T:PSL2 vs GL2}
Let $\Lambda$ be a non-zero prime ideal of $\OO$ such that $\bbar\rho_\Lambda^\proj(G)$ is conjugate to $\PSL_2(\FF_\lambda)$ or $\PGL_2(\FF_\lambda)$, where $\lambda=\Lambda\cap R$.   After conjugating $\bbar\rho_\Lambda$, we may assume that $\bbar\rho_\Lambda^\proj(G) \subseteq \PGL_2(\FF_\lambda)$.  Let $\ell$ be the rational prime lying under $\Lambda$.
\begin{romanenum}
\item \label{T:PSL2 vs GL2 i}
If $k$ is odd, then $\bbar\rho_{\Lambda}^\proj (G) = \PSL_2(\FF_\lambda)$ if and only if $\lambda$ splits completely in $L$.
\item \label{T:PSL2 vs GL2 ii}
If $k$ is even and $[\FF_\lambda:\FF_\ell]$ is even, then $\bbar\rho_{\Lambda}^\proj (G) = \PSL_2(\FF_\lambda)$ if and only if $\lambda$ splits completely in $L$.
\item \label{T:PSL2 vs GL2 iii}
If $k$ is even, $[\FF_\lambda:\FF_\ell]$ is odd, and $\ell\nmid N$, then $\bbar\rho_{\Lambda}^\proj (G) = \PGL_2(\FF_\lambda)$.\end{romanenum}
\end{theorem}

\begin{remark} 
\begin{romanenum}
\item
From Theorem~\ref{T:PSL2 vs GL2}, we see that it is more challenging to produce Galois extensions of $\QQ$ with Galois group $\PSL_2(\FF_{\ell^r})$ with odd $r$ if we focus solely on newforms with $k$ even.    However, it is still possible to obtain such groups in the excluded case $\ell | N$.
\item
Parts (\ref{T:PSL2 vs GL2 ii}) and (\ref{T:PSL2 vs GL2 iii}) of Theorem~\ref{T:PSL2 vs GL2} are included for completeness, see \cite{MR1879665}*{Proposition~1.5} for an equivalent version in the case $k=2$ due to Dieulefait.   Surprisingly, there has been very little attention in the literature given to the case where $k$ is odd (commenting on a preprint of this work, Dieulefait has shared several explicit examples worked out with Tsaknias and Vila).   In \S\ref{SS:weight 3 level 27} and \ref{SS:weight 3 level 160}, we give examples with $k=3$ and $L=K$ (so $\lambda$ splits in $L$ for any $\lambda$).
\end{romanenum}
\end{remark}

\subsection{An example realizing the groups $\text{PSL}_2(\FF_\ell)$}  \label{SS:weight 3 level 27}

We now give an example that realizes the simple groups $\PSL_2(\FF_\ell)$  as Galois groups of an extension of $\QQ$ for all primes $\ell\geq 7$.   Let $f=\sum_{n=1}^\infty a_n q^n$ be a non-CM newform of weight $3$, level $N=27$ and nebentypus $\varepsilon(a)=\big(\frac{-3}{a}\big)$.   We can choose $f$ so that\footnote{More explicitly, take
	$f=\tfrac{i}{2} g \theta_0 - \tfrac{1+i}{2}  g \theta_1 + \tfrac{3}{2} g \theta_2$, 
	where $g:= q \prod_{n \geq 1} (1-q^{3n})^2 (1-q^{9n})^2$ and  $\theta_j:= \sum_{x,y \in \ZZ} q^{3^j(x^2+xy+y^2)}$, cf.~\cite{MR885783}*{p.~228}.    }
\begin{align*}
f   =q &+ 3iq^2 - 5q^4 - 3iq^5 + 5q^7 - 3iq^8 + 9q^{10} - 15iq^{11} - 10q^{13} + \cdots;
\end{align*}
the other possibility for $f$ is its complex conjugate $\sum_n \bbar{a}_n q^n$.   

The subfield $E$ of $\CC$ generated by the coefficients $a_n$ is $\QQ(i)$.   Take any prime $p\neq 3$.  We will see that  $\bbar{a}_p = \varepsilon(p)^{-1} a_p$.   Therefore, $a_p$ or $i a_p$ belongs to $\ZZ$ when $\varepsilon(p)$ is $1$ or $-1$, respectively, and hence $r_p=a_p^2/\varepsilon(p)$ is a square in $\ZZ$.   Therefore, $L=K=\QQ$.

 In \S\ref{SS:3 27}, we shall verify that Theorem~\ref{T:Ribet} holds with $S=\{2,3,5\}$.   Take any prime $\ell \geq 7$ and prime $\Lambda\subseteq\ZZ[i]$ dividing $\ell$.  
 Theorem~\ref{T:PSL2 vs GL2} with $L=K=\QQ$ implies that $\bbar\rho^\proj_\Lambda(G)$ is isomorphic to $\PSL_2(\FF_\ell)$.    The following theorem is now an immediate consequence (it is easy to prove directly for the group $\PSL_2(\FF_5)\cong A_5$).

\begin{theorem} \label{T:example Serre}
For each prime $\ell\geq 5$, there is a Galois extension $K/\QQ$ such that $\Gal(K/\QQ)$ is isomorphic to the simple group $\PSL_2(\FF_\ell)$.
\qed
\end{theorem}

\begin{remark}
\begin{romanenum}
\item
In \S5.5 of \cite{MR885783},  J-P.~Serre describes the image of $\bbar\rho_{(7)}$ and proves that it gives rise to a $\PSL_2(\FF_7)$-extension of $\QQ$, however, he does not consider the image modulo other primes.   Note that Serre was actually giving an example of his conjecture, so he started with the $\PSL_2(\FF_7)$-extension and then found the newform $f$. 
\item
Theorem~\ref{T:example Serre} was first proved by the author in \cite{Zywina-PSL2} by considering the Galois action on the second $\ell$-adic \'etale cohomology of a specific surface.   One can show that the Galois extensions of \cite{Zywina-PSL2} could also be constructed by first starting with an appropriate newform of weight $3$ and level $32$.
\end{romanenum}
\end{remark}

\subsection{Another example}  \label{SS:weight 3 level 160}
We now give an example with $K\neq \QQ$.   Additional details will be provided in \S\ref{SS:3 160}.  Let $f=\sum_n a_n q^n$ be a non-CM newform of weight $3$, level $N=160$ and nebentypus $\varepsilon(a)=\big(\frac{-5}{a}\big)$.  

Take $E$, $K$, $L$, $R$ and $\OO$ as in \S\ref{SS:general result}.   We will see in \S\ref{SS:3 160} that $E=K(i)$ and that $K$ is the unique cubic field in $\QQ(\zeta_{13})$.  We will also observe that $L=K$. 

Take any {odd} prime $\ell$ congruent to $\pm 2$, $\pm 3$, $\pm 4$ or $\pm 6$ modulo $13$.  Let $\Lambda$ be any prime ideal of $\OO$ dividing $\ell$ and set $\lambda= \Lambda \cap R$.    The assumption on $\ell$ modulo $13$ implies that $\lambda= \ell R$ and that $\FF_\lambda \cong \FF_{\ell^3}$.    In \S\ref{SS:3 160}, we shall compute a set $S$ as in Theorem~\ref{T:Ribet} which does not contain $\lambda$.   Theorem~\ref{T:PSL2 vs GL2} with $L=K$ implies that $\bbar\rho^\proj_\Lambda(G)$ is isomorphic to $\PSL_2(\FF_\lambda)\cong \PSL_2(\FF_{\ell^3})$.   The following is an immediate consequence.

\begin{theorem} 
If $\ell$ is an odd prime congruent to $\pm 2$, $\pm 3$, $\pm 4$ or $\pm 6$ modulo $13$, then the simple group $\PSL_2(\FF_{\ell^3})$ occurs as the Galois group of an extension of $\QQ$.
\end{theorem}

\subsection*{Acknowledgements} 
Thanks to Henri Darmon for pushing the author to find the modular interpretation of the Galois representations in \cite{Zywina-PSL2}.   Thanks also to Ravi Ramakrishna and Luis Dieulefait for their comments and corrections.   Computations were performed with \texttt{Magma} \cite{Magma}.

\section{The fields $K$ and $L$} \label{S:inner twists}

Take a newform $f$ with notation and assumptions as in \S\ref{SS:general result}.   

\subsection{The field $K$} 
\label{SS:inner twists}

Let $\Gamma$ be the set of automorphisms $\gamma$ of the field $E$ for which there is a primitive Dirichlet character $\chi_\gamma$ that satisfies
\begin{equation} \label{E:ap twist}
\gamma(a_p) = \chi_\gamma(p) a_p
\end{equation}
for all primes $p\nmid N$. The set of primes $p$ with $a_p\neq 0$ has density $1$ since $f$ is non-CM, so the image of $\chi_\gamma$ lies in $E^\times$ and the character $\chi_\gamma$ is uniquely determined from $\gamma$.  

Define $M$ to be $N$ or $4N$ if $N$ is odd or even, respectively.   The conductor of $\chi_\gamma$ divides $M$, cf.~\cite{MR617867}*{Remark~1.6}.  Moreover, there is a quadratic Dirichlet character $\alpha$ with conductor dividing $M$ and an integer $i$ such that $\chi_\gamma$ is the primitive character coming from $\alpha \varepsilon^i$, cf.~\cite{MR617867}*{Lemma~1.5(i)}.  

For each prime $p\nmid N$, we have $\bbar{a}_p = \varepsilon(p)^{-1} a_p$, cf.~\cite{MR0453647}*{p.~21}, so complex conjugation induces an automorphism $\gamma$ of $E$ and $\chi_\gamma$ is the primitive character coming from $\varepsilon$.   In particular, $\Gamma \neq 1$ if $\varepsilon$ is non-trivial.

\begin{remark}
More generally, we could have instead considered an embedding $\gamma\colon E \to \CC$ and a Dirichlet character $\chi_\gamma$ such that (\ref{E:ap twist}) holds for all sufficiently large primes $p$.   This gives the same twists, since $\gamma(E)=E$ and the character $\chi_\gamma$ is unramified at primes $p\nmid N$, cf.~\cite{MR617867}*{Remark~1.3}. 
\end{remark}

  The set $\Gamma$ is in fact an abelian subgroup of $\Aut(E)$, cf.~\cite{MR617867}*{Lemma~1.5(ii)}.  Denote by $E^{\Gamma}$ the fixed field of $E$ by $\Gamma$.   

 \begin{lemma} \label{L:fields agree}
 \begin{romanenum}
 \item \label{L:fields agree i}
We have $K=E^\Gamma$ and hence $\Gal(E/K)=\Gamma$.
 \item \label{L:fields agree ii}
 There is a prime $p\nmid N$ such that $K=\QQ(r_p)$.
 \end{romanenum}
 \end{lemma}
 \begin{proof} 
Take any $p\nmid N$.  For each $\gamma\in \Gamma$, we have
 \[
 \gamma(r_p) = \gamma(a_p^2)/\gamma(\varepsilon(p)) = \chi_\gamma(p)^2 a_p^2/\gamma(\varepsilon(p)) = a_p^2/\varepsilon(p) = r_p,
 \] 
where we have used that $\chi_\gamma(p)^2 = \gamma(\varepsilon(p))/\varepsilon(p)$, cf.~\cite{MR617867}*{proof of Lemma~1.5(ii)}.  This shows that $r_p$ belong in $E^\Gamma$ and hence $K\subseteq E^\Gamma$ since $p\nmid N$ was arbitrary.    To complete the proof of the lemma, it thus suffices to show that $E^\Gamma=\QQ(r_p)$ for some prime $p\nmid N$.
 
For $\gamma\in \Gamma$, let $\widetilde\chi_\gamma\colon G \to \CC^\times$ be the continuous character such that $\widetilde\chi_\gamma(\Frob_p)=\chi_\gamma(p)$ for all $p\nmid N$.  Define the group $H = \bigcap_{\gamma\in \Gamma} \ker \widetilde\chi_\gamma$; it is an open normal subgroup of $G$ with $G/H$ is abelian.   Let $\calK$ be the subfield of $\Qbar$ fixed by $H$; it is a finite abelian extension of $\QQ$.   

Fix a prime $\ell$ and a prime ideal $\Lambda | \ell$ of $\OO$.  In the proof of Theorem~3.1 of \cite{MR819838}, Ribet proved that $E^\Gamma=\QQ(a_v^2)$ for a positive density set of finite place $v\nmid N\ell$ of $\calK$, where $a_v:=\tr(\rho_\Lambda(\Frob_v))$.    There is thus a finite place $v\nmid N\ell$ of $\calK$ of degree $1$ such that $E^\Gamma=\QQ(a_v^2)$.   We have $a_v=a_p$, where $p$ is the rational prime that $v$ divides, so $E^\Gamma=\QQ(a_p^2)$.      Since $v$ has degree $1$ and $\calK/\QQ$ is abelian, the prime $p$ must split completely in $\calK$ and hence $\chi_\gamma(p)=1$ for all $\gamma\in \Gamma$; in particular, $\varepsilon(p)=1$.   Therefore, $E^\Gamma=\QQ(r_p)$.
 \end{proof}

\subsection{The field $L$}

 Recall that we defined $L$ to be the extension of $K$ in $\CC$ obtained by adjoining the square root of $r_p = a_p^2/\varepsilon(p)$ for all $p\nmid N$.    The following allows one to find a finite set of generators for the extension $L/K$ and gives a way to check the criterion of Theorem~\ref{T:PSL2 vs GL2}.
 
 \begin{lemma} \label{L:L description}
\begin{romanenum}
 \item \label{L:L description i}
Choose primes $p_1,\ldots, p_m \nmid N$ that generate the group $(\ZZ/M\ZZ)^\times$ and satisfy $r_{p_i}\neq 0$ for all $1\leq i\leq m$.  Then $L=K(\sqrt{r_{p_1}},\ldots,\sqrt{r_{p_m}})$.
 \item \label{L:L description ii}
 Take any non-zero prime ideal $\lambda$ of $R$ that does not divide $2$.  Let $p_1,\ldots, p_m$ be primes as in (\ref{L:L description i}).  Then the following are equivalent:
 \begin{alphenum}
 \item $
 \lambda$ splits completely in $L$,
 \item 
 for all $p\nmid N$, $r_p$ is a square in $K_\lambda$,
 \item 
  for all $1\leq i \leq m$, $r_{p_i}$ is a square in $K_\lambda$.
 \end{alphenum}
 \end{romanenum}
 \end{lemma}
\begin{proof}
Take any prime $p\nmid N$.   To prove part (\ref{L:L description i}), it suffices to show that $\sqrt{r_p}$ belongs to the field $L':=K(\sqrt{r_{p_1}},\ldots,\sqrt{r_{p_m}})$.   This is obvious if $r_p=0$, so assume that $r_p\neq 0$.  Since the $p_i$ generate $(\ZZ/M\ZZ)^\times$ by assumption, there are integers $e_i\geq 0$ such that $p\equiv p_1^{e_1} \cdots p_m^{e_m} \pmod{M}$.   Take any $\gamma\in \Gamma$.  Using that the conductor of $\chi_\gamma$ divides $M$ and  (\ref{E:ap twist}), we have
\[
\gamma\Big( \frac{a_p}{{\prod}_i a_{p_i}^{e_i}} \Big) =\frac{\chi_\gamma(p)}{\chi_\gamma({\prod}_i p_i^{e_i})} \cdot \frac{a_p}{{\prod}_i a_{p_i}^{e_i}} = \frac{\chi_\gamma(p)}{\chi_\gamma(p)}\cdot \frac{a_p}{{\prod}_i a_{p_i}^{e_i}} = \frac{a_p}{{\prod}_i a_{p_i}^{e_i}},
\]
Since $E^\Gamma =K$ by Lemma~\ref{L:fields agree}(\ref{L:fields agree i}), the value $a_p/{\prod}_i a_{p_i}^{e_i}$ belongs to $K$; it is non-zero since $r_p\neq 0$ and $r_{p_i}\neq 0$.     We have $\varepsilon(p) = \prod_i \varepsilon(p_i)^{e_i}$ since the conductor of $\varepsilon$ divides $M$.   Therefore,
\[
\frac{r_p}{{\prod}_i r_{p_i}^{e_i}} = \frac{a_p^2}{{\prod}_i (a_{p_i}^2)^{e_i}} = \bigg(\frac{a_p}{{\prod}_i a_{p_i}^{e_i}}\bigg)^2 \in (K^\times)^2.
\]
This shows that $\sqrt{r_p}$ is contained in $L'$ as desired.  This proves (\ref{L:L description i}); part (\ref{L:L description ii}) is an easy consequence of (\ref{L:L description i}).
\end{proof}

\begin{remark}
Finding primes $p_i$ as in Lemma~\ref{L:L description}(\ref{L:L description i}) is straightforward since $r_p\neq 0$ for all $p$ outside a set of density $0$ (and the primes representing each class $a\in (\ZZ/M\ZZ)^\times$ have positive density).  Lemma~\ref{L:L description}(\ref{L:L description ii}) gives a straightforward way to check if $\lambda$ splits completely in $L$.   Let $e_i$ be the $\lambda$-adic valuation of $r_{p_i}$ and let $\pi$ be a uniformizer of $K_\lambda$; then $r_{p_i}$ is a square in $K_\lambda$ if and only if $e$ is even and the image of $r_{p_i}/\pi^{e_i}$ in $\FF_\lambda$ is a square.
\end{remark}

\section{Proof of Theorem~\ref{T:PSL2 vs GL2}}
We may assume that $\bbar\rho_\Lambda^\proj(G)$ is $\PSL_2(\FF_\lambda)$ or $\PGL_2(\FF_\lambda)$.   For any $n\geq 1$, the group $\GL_2(\FF_{2^n})$ is generated by $\SL_2(\FF_{2^n})$ and its scalar matrices, so $\PSL_2(\FF_{2^n})=\PGL_2(\FF_{2^n})$.   The theorem is thus trivial when $\ell=2$, so we may assume that $\ell$ is odd.

Take any $\alpha\in \PGL_2(\FF_\lambda)\subseteq \PGL_2(\FF_\Lambda)$ and choose any matrix $A\in \GL_2(\FF_\Lambda)$ whose image in $\PGL_2(\FF_\Lambda)$ is $\alpha$.  The value $\tr(A)^2/\det(A)$ does not depend on the choice of $A$ and lies in $\FF_\lambda$ (since we can choose $A$ in $\GL_2(\FF_\lambda)$); by abuse of notation, we denote this common value by $\tr(\alpha)^2/\det(\alpha)$.

\begin{lemma}  \label{L:PSL2 p criterion}
Suppose that $p\nmid N\ell$ is a prime for which $r_p \not\equiv 0\pmod{\lambda}$.   Then $\bbar\rho_\Lambda^\proj(\Frob_p)$ is contained in $\PSL_2(\FF_\lambda)$ if and only if the image of $a_p^2/(\varepsilon(p) p^{k-1})=r_p/p^{k-1}$ in $\FF_\lambda^\times$ is a square.
\end{lemma}
\begin{proof}
Define $A:= \bbar\rho_{\Lambda}(\Frob_p)$ and $\alpha:=\bbar\rho_\Lambda(\Frob_p)$; the image of $A$ in $\PGL_2(\FF_\Lambda)$ is $\alpha$.   The value $\xi_p:=\tr(\alpha)^2/\det(\alpha) = \tr(A)^2/\det(A)$ agrees with the image of $a_p^2/(\varepsilon(p)p^{k-1})= r_p/p^{k-1}$ in $\FF_\Lambda$.     Since $r_p \in R$ is non-zero modulo $\lambda$ by assumption, the value $\xi_p$ lies in $\FF_\lambda^\times$.    Fix a matrix $A_0 \in \GL_2(\FF_\lambda)$ whose image in $\PGL_2(\FF_\lambda)$ is $\alpha$;  we have $\xi_p = \tr(A_0)^2/\det(A_0)$.   Since $\xi_p \neq 0$, we find that $\xi_p$ and $\det(A_0)$ lies in the same coset in $\FF_\lambda^\times/(\FF_\lambda^\times)^2$.

The determinant gives rise to a homomorphism $d\colon \PGL_2(\FF_\lambda) \to \FF_\lambda^\times/( \FF_\lambda^\times)^2$ whose kernel is $\PSL_2(\FF_\lambda)$.   Define the character 
\[
\xi \colon G \xrightarrow{\bbar\rho_\Lambda^\proj} \PGL_2(\FF_\lambda) \xrightarrow{d} \FF_\lambda^\times/(\FF_\lambda^\times)^2.
\]
We have $\xi(\Frob_p) = \det(A_0) \cdot ( \FF_\lambda^\times)^2 = \xi_p \cdot ( \FF_\lambda^\times)^2$.   So $\xi(\Frob_p)=1$, equivalently $\bbar\rho_\Lambda^\proj(\Frob_p) \in \PSL_2(\FF_\lambda)$, if and only if $\xi_p \in \FF_\lambda^\times$ is a square.  
\end{proof}

Let $M$ be the integer from \S\ref{SS:inner twists}.

\begin{lemma} \label{L:rp progressions}
For each $a\in (\ZZ/M\ell\ZZ)^\times$, there is a prime $p\equiv a \pmod{M\ell}$ such that $r_p \not\equiv 0 \pmod{\lambda}$.
\end{lemma}
\begin{proof}
Set $H= \bbar\rho_\Lambda^\proj(G)$; it is $\PSL_2(\FF_\lambda)$ or $\PGL_2(\FF_\lambda)$ by assumption.  Let $H'$ be the commutator subgroup of $H$.   We claim that  for each coset $\kappa$ of $H'$ in $H$, there exists an $\alpha\in \kappa$ with $\tr(\alpha)^2/\det(\alpha)\neq 0$.   If $H'=\PSL_2(\FF_\lambda)$, then the claim is easy; note that for any $t\in \FF_\lambda$ and $d\in \FF_\lambda^\times$, there is a matrix in $\GL_2(\FF_\lambda)$ with trace $t$ and determinant $d$.  When $\#\FF_\lambda\neq 3$, the group $\PSL_2(\FF_\lambda)$ is non-abelian and simple, so $H'=\PSL_2(\FF_\lambda)$.  When $\#\FF_\lambda=3$ and $H=\PGL_2(\FF_\lambda)$, we have $H'=\PSL_2(\FF_\lambda)$.    It thus suffices to prove the claim in the case where $\FF_\lambda=\FF_3$ and $H=\PSL_2(\FF_3)$.   In this case, $H'$ is the unique subgroup of $H$ of index $3$ and the cosets of $H/H'$ are represented by $\left(\begin{smallmatrix} 1 & b \\  0 & 1 \end{smallmatrix}\right)$ with $b\in \FF_3$.   The claim is now immediate in this remaining case.

Let $\chi \colon G\twoheadrightarrow (\ZZ/M\ell\ZZ)^\times$ be the cyclotomic character that satisfies $\chi(\Frob_p) \equiv p \pmod{M\ell}$ for all $p\nmid M\ell$.    The set $\bbar\rho_\Lambda(\chi^{-1}(a))$ is thus the union of cosets of $H'$ in $H$.    By the claim above, there exists an $\alpha\in \bbar\rho_\Lambda^\proj(\chi^{-1}(a))$ with $\tr(\alpha)^2/\det(\alpha)\neq 0$.   By the Chebotarev density theorem, there is a prime $p\nmid M\ell$ satisfying $p\equiv a \pmod{M\ell}$ and  $\bbar\rho_\Lambda^\proj(\Frob_p)= \alpha$.   The lemma follows since $r_p/p^{k-1}$ modulo $\lambda$ agrees with $\tr(\alpha)^2/\det(\alpha)\neq 0$.   
\end{proof}

\noindent \textbf{Case 1:} \emph{Assume that $k$ is odd or $[\FF_\lambda:\FF_\ell]$ is even.}

First suppose that $\bbar\rho_\Lambda^\proj(G)=\PSL_2(\FF_\lambda)$.  By Lemma~\ref{L:rp progressions}, there are primes $p_1,\ldots, p_m \nmid N\ell$ that generate the group $(\ZZ/M\ZZ)^\times$ and satisfy $r_{p_i}\not\equiv 0 \pmod{\lambda}$ for all $1\leq i\leq m$.   By Lemma~\ref{L:PSL2 p criterion} and the assumption $\bbar\rho_\Lambda^\proj(G)=\PSL_2(\FF_\lambda)$, the image of $r_{p_i}/{p_i}^{k-1}$ in $\FF_\lambda$ is a non-zero square for all $1\leq i \leq m$.     For each $1\leq i \leq m$, the assumption that $k$ is odd or $[\FF_\lambda:\FF_\ell]$ is even implies that $p_i^{k-1}$ is a square in $\FF_\lambda$ and hence the image of $r_{p_i}$ in $\FF_\lambda$ is a non-zero square.    Since $\lambda \nmid 2$, we deduce that each $r_{p_i}$ is a square in $K_\lambda$.   By Lemma~\ref{L:L description}(\ref{L:L description ii}), the prime $\lambda$ splits completely in $L$.

Now suppose that $\bbar\rho_\Lambda^\proj(G)=\PGL_2(\FF_\lambda)$.   There exists an element $\alpha \in \PGL_2(\FF_\lambda) - \PSL_2(\FF_\lambda)$ with $\tr(\alpha)^2/\det(\alpha) \neq 0$.    By the Chebotarev density theorem, there is a prime $p\nmid N\ell$ such that $\bbar\rho_\Lambda^\proj(\Frob_p) = \alpha$.   We have $r_p \equiv \tr(\alpha)^2/\det(\alpha) \not\equiv 0\pmod{\lambda}$.   Since $\bbar\rho_\Lambda^\proj(\Frob_p) \notin \PSL_2(\FF_\lambda)$, Lemma~\ref{L:PSL2 p criterion} implies that the image of $r_p/p^{k-1}$ in $\FF_\lambda$ is not a square.   Since $k$ is odd or $[\FF_\lambda:\FF_\ell]$ is even, the image of $r_p$ in $\FF_\lambda$ is not a square.  Therefore, $r_p$ is not a square in $K_\lambda$.   By Lemma~\ref{L:L description}(\ref{L:L description ii}), we deduce that $\lambda$ does not split completely in $L$.
\\

\noindent \textbf{Case 2:} \emph{Assume that $k$ is even, $[\FF_\lambda:\FF_\ell]$ is odd, and $\ell\nmid N$.}

Since $\ell\nmid N$, there is an integer $a\in \ZZ$ such that $a\equiv 1 \pmod{M}$ and $a$ is not a square modulo $\ell$.   By Lemma~\ref{L:rp progressions}, there is a prime $p\equiv a \pmod{M\ell}$ such that $r_p\not\equiv 0 \pmod{\lambda}$.  

We claim that $a_p\in R$ and $\varepsilon(p)=1$.  With notation as in \S\ref{SS:inner twists}, take any $\gamma\in \Gamma$.   Since the conductor of $\chi_\gamma$ divides $M$ and $p\equiv 1 \pmod{M}$, we have $\gamma(a_p)=\chi_\gamma(p) a_p =a_p$.   Since $\gamma\in \Gamma$ was arbitrary, we have $a_p \in K$ by Lemma~\ref{L:fields agree}.  Therefore, $a_p\in R$ since it is an algebraic integer.  We have $\varepsilon(p)=1$ since $p\equiv 1\pmod{N}$.

Since $a_p\in R$ and $r_p\not\equiv 0 \pmod{\lambda}$, the image of $a_p^2$ in $\FF_\lambda$ is a non-zero square.    Since $k$ is even, $p^k$ is a square in $\FF_\lambda$.  Since $p$ is not a square modulo $\ell$ and $[\FF_\lambda:\FF_\ell]$ is odd, the prime $p$ is not a square in $\FF_\lambda$.  So the image of 
\[
a_p^2/(\varepsilon(p)p^{k-1}) = p\cdot a_p^2/p^k
\]
in $\FF_\lambda$ is not a square.    Lemma~\ref{L:PSL2 p criterion} implies that $\bbar\rho_\Lambda^\proj(\Frob_p) \notin \PSL_2(\FF_\lambda)$.   Therefore, $\bbar\rho_\Lambda^\proj(G)=\PGL_2(\FF_\lambda)$.

\section{An effective version of Theorem~\ref{T:Ribet}}  \label{S:effective Ribet}

Take a newform $f$ with notation and assumptions as in \S\ref{SS:general result}.     Let $\lambda$ be a non-zero prime ideal of $R$ and let $\ell$ be the prime lying under $\lambda$.   Let $k_\lambda$ be the subfield of $\FF_\lambda$ generated by the image of $r_p$ modulo $\lambda$ with primes $p\nmid N\ell$.   Take any prime ideal $\Lambda$ of $\OO$ that divides $\lambda$.

In this section, we describe how to compute an explicit finite set $S$ of prime ideals of $R$ as in Theorem~\ref{T:Ribet}.   First some simple definitions:
\begin{itemize}
\item
Let $\FF$ be an extension of $\FF_\Lambda$ of degree $\gcd(2,\ell)$. 
\item
Let $e_0=0$ if $\ell\geq k-1$ and $\ell\nmid N$, and $e_0=\ell-2$ otherwise.
\item
Let $e_1=0$ if $N$ is odd, and $e_1=1$ otherwise.   
\item
Let $e_2=0$ if $\ell\geq 2k$, and $e_2=1$ otherwise.   
\item
Define $\calM=4^{e_1} \ell^{e_2}\prod_{p|N} p$.
\end{itemize}
We will prove the following in \S\ref{S:proof theorem criterion}.

\begin{theorem} \label{T:criteria}
Suppose that all the following conditions hold:
\begin{alphenum}
\item \label{P:criteria a}
For every integer $0\leq j \leq e_0$ and character $\chi\colon (\ZZ/N\ZZ)^\times \to \FF^\times$, there is a prime $p\nmid N\ell$ such that $\chi(p)p^j \in \FF$ is not a root of the polynomial $x^2 - a_px + \varepsilon(p)p^{k-1} \in \FF_\Lambda[x]$.

\item \label{P:criteria b}
For every non-trivial character $\chi \colon (\ZZ/\calM\ZZ)^\times \to \{\pm 1\}$, there is a prime $p\nmid N\ell$ such that $\chi(p)=-1$ and $r_p \not\equiv 0 \pmod{\lambda}$.

\item \label{P:criteria c}
If $\#k_\lambda \notin \{4,5\}$, then at least one of the following hold:
\begin{itemize}
\item
 $\ell > 5k-4$ and $\ell \nmid N$, 
\item
$\ell \equiv 0, \pm 1 \pmod{5}$ and  $\#k_\lambda\neq\ell$,
\item
$\ell \equiv \pm 2\pmod{5}$ and $\#k_\lambda\neq \ell^2$,
\item
there is a prime $p\nmid N\ell$ such that the image of $a_{p}^2/(\varepsilon(p) p^{k-1})$ in $\FF_\lambda$ is not equal to $0$, $1$ and $4$, and is not a root of $x^2-3x+1$. 
\end{itemize}

\item  \label{P:criteria d}
If $\#k_\lambda \notin \{3,5,7\}$, then at least one of the following hold:
\begin{itemize}
\item
 $\ell > 4k-3$ and $\ell \nmid N$, 
\item
$\#k_\lambda\neq \ell$,
\item
there is a prime $p\nmid N\ell$ such that the image of $a_{p}^2/(\varepsilon(p) p^{k-1})$ in $\FF_\lambda$ is not equal to $0$, $1$, $2$ and $4$.  
\end{itemize}

\item \label{P:criteria e}
If $\#k_\lambda \in \{5,7\}$, then for every non-trivial character $\chi \colon (\ZZ/ 4^{e_1}\ell N\ZZ)^\times\to \{\pm 1\}$ there is a prime $p\nmid N\ell$ such that $\chi(p)=1$ and $a_{p}^2/(\varepsilon(p) p^{k-1})\equiv 2 \pmod{\lambda}$.
\end{alphenum}

Then the group $\bbar\rho_\Lambda^\proj(G)$ is conjugate in $\PGL_2(\FF_\Lambda)$ to $\PSL_2(k_\lambda)$ or $\PGL_2(k_\lambda)$.
\end{theorem}

\begin{remark}
Note that the above conditions simplify greatly if one also assumes that $\ell\nmid N$ and $\ell>5k-4$.    

Though we will not prove it, Theorem~\ref{T:criteria} has been stated so that all the conditions (\ref{P:criteria a})--(\ref{P:criteria e}) hold if and only if $\bbar\rho_\Lambda^\proj(G)$ is conjugate to $\PSL_2(k_\lambda)$ or $\PGL_2(k_\lambda)$.    In particular, after considering enough primes $p$, one will obtain the minimal set $S$ of Theorem~\ref{T:Ribet} (one could use an effective version of Chebotarev density to make this a legitimate algorithm).
\end{remark}

Let us now describe how to compute a set of exceptional primes as in  Theorem~\ref{T:Ribet}.  Define $M=N$ if $N$ is odd and $M=4N$ otherwise.  Set  $\calM':=4^{e_1} \prod_{p|N} p$.    We first choose some primes:
\begin{itemize}
\item
Let $q_1,\ldots, q_n$ be primes congruent to $1$ modulo $N$.
\item
Let $p_1,\ldots, p_m \nmid N$ be primes with $r_{p_i}\neq 0$ such that for every non-trivial character $\chi \colon (\ZZ/\calM' \ZZ)^\times \to \{\pm 1\}$, we have $\chi(p_i)=-1$ for some $1\leq i \leq m$.
\item
Let $p_0 \nmid N$ be a prime such that $\QQ(r_{p_0})=K$.  
\end{itemize}
That such primes $p_1,\ldots, p_m$ exist is clear since the set of primes $p$ with $r_p\neq 0$ has density $1$.   That such a prime $q$ exists follows from Lemma~\ref{L:fields agree} (the set of such $q$ actually has positive density).   Define the ring $R':=\ZZ[a_{q}^2/\varepsilon(q)]$; it is an order in $R$.  

Define $S$ to be the set of non-zero primes $\lambda$ of $R$, dividing a rational prime $\ell$, that satisfy one of the following conditions:
\begin{itemize}
\item 
$\ell \leq 5k-4$ or $\ell\leq 7$,
\item
$\ell | N$,
\item
for all $1\leq i \leq n$, we have $\ell=q_i$ or $r_{q_i} \equiv (1+q_i^{k-1})^2 \pmod{\lambda}$,
\item
for some $1\leq i \leq m$, we have $\ell=p_i$ or $r_{p_i} \equiv 0 \pmod{\lambda}$,
\item
$\ell=q$ or $\ell$ divides $[R:R']$.
\end{itemize}
Note that the set $S$ is \emph{finite} (the only part that is not immediate is that $r_{q_i} - (1+q_i^{k-1})^2\neq 0$; this follows from Deligne's bound $|r_{q_i}|=|a_{q_i}| \leq 2q_i^{(k-1)/2}$ and $k>1$).  The following is our effective version of Theorem~\ref{T:Ribet}.

\begin{theorem}  \label{T:effective version}
Take any non-zero prime ideal $\lambda \notin S$ of $R$ and let $\Lambda$ be any prime of $\OO$ dividing $\lambda$.  
Then the group $\bbar\rho_\Lambda^\proj(G)$ is conjugate in $\PGL_2(\FF_\Lambda)$ to either $\PSL_2(\FF_\lambda)$ or $\PGL_2(\FF_\lambda)$.
\end{theorem}
\begin{proof}
Let $\ell$ be the rational prime lying under $\lambda$.     We shall verify the conditions of Theorem~\ref{T:criteria}.

We first show that condition (\ref{P:criteria a}) of Theorem~\ref{T:criteria} holds.  Take any integer $0\leq j \leq e_0$ and character $\chi\colon (\ZZ/N\ZZ)^\times \to \FF^\times=\FF_\Lambda^\times$.    We have $\ell>5k-4>k-1$ and $\ell \nmid N$ since $\lambda \notin S$,  so $e_0=0$ and hence $j=0$.    Take any $1\leq i \leq n$.  Since $q_i\equiv 1\pmod{N}$ and $j=0$, we have $\chi(q_i)q_i^j=1$ and $\varepsilon(q_i)=1$.     Since $\lambda \notin S$, we also have $q_i \nmid N \ell$ ($q_i\nmid N$ is immediate from the congruence imposed on $q_i$).   If $\chi(q_i)q_i^j=1$ was a root of $x^2-a_{q_i}x+\varepsilon(q_i)q_i^{k-1}$ in $\FF_\Lambda[x]$, then we would have $a_{q_i} \equiv 1 + q_i^{k-1} \pmod{\Lambda}$; squaring and using that $\varepsilon(q_i)=1$, we deduce that $r_{q_i} \equiv (1+q_i^{k-1})^2 \pmod{\lambda}$.   Since $\lambda\notin S$, we have $r_{q_i} \not \equiv (1+q_i^{k-1})^2 \pmod{\lambda}$ for some $1\leq i \leq n$ and hence $\chi(q_i)q_i^j$ is not a root of $x^2-a_{q_i}x+\varepsilon(q_i)q_i^{k-1}$.

We now show that condition (\ref{P:criteria b}) of Theorem~\ref{T:criteria} holds.   We have $e_2=0$ since $\lambda\notin S$, and hence $\calM'=\calM$.  Take any non-trivial character $\chi \colon (\ZZ/\calM\ZZ)^\times \to \{\pm 1\}$.  By our choice of primes $p_1,\ldots,p_m$, we have $\chi(p_i)=-1$ for some $1\leq i \leq m$.   The prime $p_i$ does not divide $N\ell$ (that $p_i\neq \ell$ follows since $\lambda\notin S$).    Since $\lambda\notin S$, we have $r_{p_i}\not\equiv 0 \pmod{\lambda}$.

Since $\lambda\notin S$, the prime $\ell\nmid N$ is greater that $7$, $4k-3$ and $5k-4$.  Conditions (\ref{P:criteria c}), (\ref{P:criteria d}) and (\ref{P:criteria e}) of Theorem~\ref{T:criteria} all hold.

Theorem~\ref{T:criteria} now implies that $\bbar\rho_\Lambda^\proj(G)$ is conjugate in $\PGL_2(\FF_\Lambda)$ to either $\PSL_2(k_\lambda)$ or $\PGL_2(k_\lambda)$.   It remains to prove that $k_\lambda=\FF_\lambda$.   We have $q\neq \ell$ since $\lambda \notin S$.   The image of the reduction map $R' \to \FF_\lambda$ thus lies in $k_\lambda$.   We have  $\ell\nmid [R:R']$ since $\lambda\notin S$, so the map $R'\to \FF_\lambda$ is surjective.   Therefore, $k_\lambda=\FF_\lambda$.
\end{proof}

\section{Proof of Theorem~\ref{T:criteria}}  \label{S:proof theorem criterion}

\subsection{Some group theory}

Fix a prime $\ell$ and an integer $r\geq 1$.  

A \defi{Borel subgroup} of $\GL_2(\FF_{\ell^r})$ is a subgroup conjugate to the subgroup of upper triangular matrices.   

A \defi{split Cartan subgroup} of $\GL_2(\FF_{\ell^r})$ is a subgroup conjugate to the subgroup of diagonal matrices.   A \defi{non-split Cartan subgroup} of $\GL_2(\FF_{\ell^r})$ is a subgroup that is cyclic of order $(\ell^r)^2-1$.  Fix a Cartan subgroup $\calC$ of $\GL_2(\FF_{\ell^r})$.   Let $\calN$ be the normalizer of $\calC$ in $\GL_2(\FF_{\ell^r})$.   One can show that $[\calN:\calC] = 2$ and that $\tr(g)=0$ and $g^2$ is scalar for all $g\in \calN-\calC$. 

\begin{lemma} \label{L:subgroups}
Fix a prime $\ell$ and an integer $r\geq 1$.   Let $G$ be a subgroup of $\GL_2(\FF_{\ell^r})$ and let $\bbar{G}$ be its image in $\PGL_2(\FF_{\ell^r})$.  Then at least one of the following hold:
\begin{enumerate}
\item
$G$ is contained in a Borel subgroup of $\GL_2(\FF_{\ell^r})$,
\item
$G$ is contained in the normalizer of a Cartan subgroup of $\GL_2(\FF_{\ell^r})$,
\item
$\bbar{G}$ is isomorphic to $\mathfrak{A}_4$,
\item
$\bbar{G}$ is isomorphic to $\mathfrak{S}_4$,
\item
$\bbar{G}$ is isomorphic to $\mathfrak{A}_5$,
\item
$\bbar G$ is conjugate to $\PSL_2(\FF_{\ell^s})$ or $\PGL_2(\FF_{\ell^s})$ for some integer $s$ dividing $r$.
\end{enumerate}
\end{lemma}
\begin{proof}
This can be deduced directly from a theorem of Dickson, cf.~\cite{MR0224703}*{Satz~8.27}, which will give the finite subgroups of $\PSL_2(\FFbar_\ell) = \PGL_2(\FFbar_\ell)$.  The finite subgroups of $\PGL_2(\FF_{\ell^r})$ have been worked out in \cite{Xander}.
\end{proof}

\begin{lemma} \label{L:easy image PGL}
Fix a prime $\ell$ and an integer $r\geq1$.   Take a matrix $A\in \GL_2(\FF_{\ell^r})$ and let $m$ be its order in $\PGL_2(\FF_{\ell^r})$.
\begin{romanenum}
\item
Suppose that $\ell \nmid m$.  If $m$ is $1$, $2$, $3$ or $4$, then $\tr(A)^2/\det(A)$ is $4$, $0$, $1$ or $2$, respectively.    If $m=5$, then $\tr(A)^2/\det(A)$ is a root of $x^2-3x+1$.
\item
If $\ell | m$, then $\tr(A)^2/\det(A) =   4$.
\end{romanenum}
\end{lemma}
\begin{proof}
The quantity $\tr(A)^2/\det(A)$ does not change if we replace $A$ by a scalar multiple or by a conjugate in $\GL_2(\FFbar_\ell)$.  If $\ell\nmid m$, then we may thus assume that $A=\left(\begin{smallmatrix} \zeta & 0 \\  0 & 1 \end{smallmatrix}\right)$ where $\zeta\in \FFbar_\ell$ has order $m$. We have $\tr(A)^2/\det(A) = \zeta+\zeta^{-1} + 2$, which is $4$, $0$, $1$ or $2$ when $m$ is $1$, $2$, $3$ or $4$, respectively.   If $m=5$, then $ \zeta+\zeta^{-1} + 2$ is a root of $x^2-3x+1$.  If $\ell| m$, then after conjugating and scaling, we may assume that $A=\left(\begin{smallmatrix} 1 & 1 \\  0 & 1 \end{smallmatrix}\right)$ and hence $\tr(A)^2/\det(A) =4$. 
\end{proof}

\subsection{Image of inertia at $\ell$}

Fix an inertia subgroup $\calI_\ell$ of $G=\Gal(\Qbar/\QQ)$ for the prime $\ell$; it is uniquely defined up to conjugacy.  The following gives important information concerning the representation $\bbar\rho_\Lambda|_{\calI_\ell}$ for large $\ell$.   Let $\chi_\ell \colon G\twoheadrightarrow \FF_\ell^\times$ be the character such that for each prime $p\nmid \ell$, $\chi_\ell$ is unramified at $p$ and $\chi_\ell(\Frob_p)\equiv p \pmod{\ell}$.
 
\begin{lemma}  \label{L:inertia}
Fix a prime $\ell \geq k-1$ for which $\ell\nmid 2N$.   Let $\Lambda$ be a prime ideal of $\OO$ dividing $\ell$ and set $\lambda=\Lambda\cap R$.
\begin{romanenum}
\item \label{L:inertia i}
Suppose that $r_\ell \not\equiv 0 \pmod{\lambda}$.  After conjugating $\bbar\rho_\Lambda$ by a matrix in $\GL_2(\FF_\Lambda)$, we have
\[
\bbar\rho_\Lambda|_{\calI_\ell} = \left(\begin{matrix} \chi_\ell^{k-1}|_{\calI_\ell} & * \\  0 & 1 \end{matrix}\right)
\]
In particular, $\bbar\rho_\Lambda^\proj(\calI_\ell)$ contains a cyclic group of order $(\ell-1)/\gcd(\ell-1,k-1)$.

\item \label{L:inertia ii}
Suppose that $r_\ell \equiv 0 \pmod{\lambda}$.   Then $\bbar\rho_\Lambda|_{\calI_\ell}$ is absolutely irreducible and $\bbar\rho_\Lambda(\calI_\ell)$ is cyclic.  Furthermore, the group $\bbar\rho_\Lambda^\proj(\calI_\ell)$ is cyclic of order $(\ell+1)/\gcd(\ell+1,k-1)$.  

\end{romanenum}
\end{lemma}
\begin{proof}
Parts (\ref{L:inertia i}) and (\ref{L:inertia ii}) follow from Theorems~2.5 and Theorem~2.6, respectively, of \cite{MR1176206}; they are theorems of Deligne and Fontaine, respectively.    We have used that $r_\ell=a_\ell^2/\varepsilon(\ell) \in R$ is congruent to $0$ modulo $\lambda$ if and only if $a_\ell\in \OO$ is congruent to $0$ modulo $\Lambda$.
\end{proof}

\subsection{Borel case}  \label{SS:Borel case 0}

Suppose that $\bbar\rho_\Lambda(G)$ is a reducible subgroup of $\GL_2(\FF)$.    There are thus characters $\psi_1,\psi_2 \colon G \to \FF^\times$ such that after conjugating the $\FF$-representation $G \xrightarrow{\bbar\rho_\Lambda} \GL_2(\FF_\Lambda) \subseteq \GL_2(\FF)$,
we have
\[
\bbar\rho_\Lambda = \left(\begin{matrix} \psi_1 & * \\  0 & \psi_2 \end{matrix}\right).
\]
The characters $\psi_1$ and $\psi_2$ are unramified at each prime $p\nmid N\ell$ since $\bbar\rho_\Lambda$ is unramified at such primes.  

\begin{lemma}  \label{L:unramified LCFT}
For each $i\in \{1,2\}$, there is a unique integer $0\leq m_i < \ell-1$ such that $\psi_i\chi_\ell^{-m_i} \colon G \to \FF^\times$ is unramified at all primes $p\nmid N$.    If $\ell \geq k-1$ and $\ell\nmid N$, then $m_1$ or $m_2$ is $0$.
\end{lemma}
\begin{proof}
The existence and uniqueness of $m_i$ is an easy consequence of class field theory for $\QQ_\ell$.   
 A choice of embedding $\Qbar\subseteq \Qbar_\ell$ induces an injective homomorphism $G_{\QQ_\ell}:=\Gal(\Qbar_\ell/\QQ_\ell) \hookrightarrow G$.   
Let $\QQ_\ell^\ab$ be the maximal abelian extension of $\QQ_\ell$ in $\Qbar_\ell$.  Restricting $\psi_i$ to $G_{\QQ_\ell}$, we obtain a representation $\psi_i\colon G_{\QQ_\ell}^\ab:=\Gal(\QQ_\ell^\ab/\QQ_\ell) \to \FF^\times$.  By local class field, the inertia subgroup $\calI$ of $G_{\QQ_\ell}^\ab$ is isomorphic to $\ZZ_\ell^\times$.    Since $\ell$ does not divide the cardinality of $\FF^\times$, we find that $\psi_i|_{\calI}$ factors through a group isomorphic to $\FF_\ell^\times$.   The character $\psi_i|_{\calI}$ must agree with a power of $\chi_\ell|_{\calI}$ since $\chi_\ell\colon G_{\QQ_\ell} \to \FF_\ell^\times$ satisfies $\chi_\ell(\calI)=\FF_\ell^\times$ and $\FF_\ell^\times$ is cyclic.

The second part of the lemma follows immediately from Lemma~\ref{L:inertia}.
\end{proof}

Take any $i\in \{1,2\}$.  By Lemma~\ref{L:unramified LCFT}, there is a unique $0\leq m_i < \ell-1$ such that the character
\[
\tilde{\psi}_i:=\psi_i\chi_\ell^{-m_i}\colon G\to \FF^\times
\] 
is unramified at $\ell$ and at all primes $p\nmid N$.  There is a character $\chi_i \colon (\ZZ/N_i\ZZ)^\times \to \FF^\times$ with $N_i\geq 1$ dividing some power of $N$ and $\ell\nmid N_i$ such that $\tilde\psi_i(\Frob_p) = \chi_i(p)$ for all $p\nmid N\ell$.   We may assume that $\chi_i$ is taken so that $N_i$ is minimal.

\begin{lemma} \label{L:Ni divides N}
The integer $N_i$ divides $N$.
\end{lemma}
\begin{proof}
We first recall the notion of an Artin conductor.
Consider a representation $\rho\colon G\to \Aut_{\FF}(V)$, where $V$ is a finite dimensional $\FF$-vector space.  Take any prime $p\neq \ell$.   A choice of embedding $\Qbar\subseteq \Qbar_p$ induces an injective homomorphism $\Gal(\Qbar_p/\QQ_p) \hookrightarrow G$.   Choose any finite Galois extension $L/\QQ_p$ for which $\rho(\Gal(\Qbar_p/L))=\{I\}$.     For each $i\geq 0$, let $H_i$ be the $i$-th ramification subgroup of $\Gal(L/\QQ_p)$ with respect to the lower numbering.   Define the integer 
\[
f_p(\rho)= \sum_{i\geq 0} [H_0:H_i]^{-1}\cdot \dim_{\FF} V/V^{H_i}.
\]
The \defi{Artin conductor} of $\rho$ is the integer $N(\rho):=\prod_{p\neq \ell} p^{f_p(\rho)}$.

Using that the character $\tilde\psi_i\colon G \to \FF^\times$ is unramified at $\ell$, one can verify that $N(\tilde\psi_i)=N_i$.   Consider our representation $\bbar\rho_\Lambda \colon G \to \GL_2(\FF)$.  For a fixed prime $p\neq \ell$, take $L$ and $H_i$ as above.  The semisimplification of $\bbar\rho_\Lambda$ is $V_1\oplus V_2$, where $V_i$ is the one dimensional representation given by $\psi_i$.     We have $f_p(\psi_1)+f_p(\psi_2) \leq f_p(\bbar\rho_\Lambda )$ since $\dim_{\FF} V^{H_i} \leq \dim_{\FF} V_1^{H_i}  + \dim_{\FF} V_2^{H_i}$.    By using this for all $p\neq \ell$, we deduce that $N(\psi_1) N(\psi_2)=N_1 N_2$ divides $N(\bbar\rho_\Lambda)$.     The lemma follows since $N(\bbar\rho_\Lambda)$ divides $N$, cf.~\cite{MR987567}*{Prop.~0.1}.
\end{proof}

Fix an $i\in \{1,2\}$;  if $\ell \geq k-1$ and $\ell\nmid N$, then we may suppose that $m_i=0$ by Lemma~\ref{L:unramified LCFT}.   Since the conductor of $\chi_i$ divides $N$ by Lemma~\ref{L:Ni divides N},  assumption (\ref{P:criteria a}) implies that there is a prime $p\nmid N\ell$ for which $\chi_i(p) p^{m_i} \in \FF$ is not a root of $x^2-a_p x + \varepsilon(p)p^{k-1} \in \FF[x]$.   However, this is a contradiction since
\[
\chi_i(p) p^{m_i} = \tilde\psi_i(\Frob_p) \chi_\ell(\Frob_p)^{m_i} = \psi_i(\Frob_p)
\]
is a root of $x^2-a_p x + \varepsilon(p)p^{k-1}$.

Therefore, the $\FF$-representation $\bbar\rho_\Lambda$ is irreducible.  In particular, $\bbar\rho_\Lambda(G)$ is not contained in a Borel subgroup of $\GL_2(\FF_\Lambda)$.

\subsection{Cartan case}

\begin{lemma} \label{L:rule out non-split Cartan}
The group $\bbar\rho_\Lambda(G)$ is not contained in a Cartan subgroup of $\GL_2(\FF_\Lambda)$.
\end{lemma}
\begin{proof}
Suppose that $\bbar\rho_\Lambda(G)$ is contained in a Cartan subgroup $\calC$ of $\GL_2(\FF_\Lambda)$.  If $\ell=2$, then $\calC$ is reducible as a subgroup of $\GL_2(\FF)$ since $\FF/\FF_\Lambda$ is a quadratic extension.  However, we saw in \S\ref{SS:Borel case 0} that $\bbar\rho_\Lambda(G) \subseteq \calC$ is an irreducible subgroup of $\GL_2(\FF)$.  Therefore, $\ell$ is odd.   If $\calC$ is split, then $\bbar\rho_\Lambda(G)$ is a reducible subgroup of $\GL_2(\FF_\Lambda)$.  This was ruled out in \S\ref{SS:Borel case 0}, so $\calC$ must be a non-split Cartan subgroup with $\ell$ odd.    

Recall that the representation $\bbar\rho_\Lambda$ is \emph{odd}, i.e., if $c\in G$ is an element corresponding to complex conjugation under some embedding $\Qbar\hookrightarrow \CC$, then $\det(\bbar\rho_\Lambda(c))=-1$.  Therefore, $\bbar\rho_\Lambda(c)$ has order $2$ and determinant $-1\neq 1$ (this last inequality uses that $\ell$ is odd).     A non-split Cartan subgroup $\calC$ of $\GL_2(\FF_\Lambda)$ is cyclic and hence $-I$ is the unique element of $\calC$ of order $2$.   Since $\det(-I)= 1$, we find that $\bbar\rho_\Lambda(c)$ does not belong to $\calC$; this gives the desired contradiction.
\end{proof}

\subsection{Normalizer of a Cartan case}  \label{SS:Cartan case}
Suppose that $\bbar\rho_\Lambda(G)$ is contained in the normalizer $\calN$ of a Cartan subgroup $\calC$ of $\GL_2(\FF_\Lambda)$.   The group $\calC$ has index $2$ in $\calN$, so we obtain a character 
\[
\beta_\Lambda\colon G \xrightarrow{\bbar\rho_\Lambda} \calN \to \calN/\calC\cong \{\pm 1\}.
\] 
The character $\beta_\Lambda$ is non-trivial since $\bbar\rho_\Lambda(G) \not\subseteq \calC$ by Lemma~\ref{L:rule out non-split Cartan}.

\begin{lemma}  \label{L:Cartan character}
The character $\beta_\Lambda$ is unramified at all primes $p\nmid N\ell$.  If $\ell \geq 2k$ and $\ell \nmid N$, then the character $\beta_\Lambda$ is also unramified at $\ell$.
\end{lemma}
\begin{proof}
The character $\beta_\Lambda$ is unramified at each prime $p\nmid N\ell$ since $\bbar\rho_\Lambda$ is unramified at such primes.  Now suppose that $\ell \geq 2k$ and $\ell\nmid N$.   We have $\ell>2$, so $\ell \nmid |\calN|$ and hence Lemma~\ref{L:inertia} implies that $\bbar\rho_\Lambda(\calI_\ell)$ is cyclic.   Moreover, Lemma~\ref{L:inertia} implies that $\bbar\rho_\Lambda^\proj(\calI_\ell)$ is cyclic  of order $d\geq (\ell-1)/(k-1)$.  Our assumption $\ell\geq 2k$ ensures that $d>2$.  

Now take a generator $g$ of $\bbar\rho_\Lambda(\calI_\ell)$.   Suppose that $\beta_\Lambda$ is ramified at $\ell$ and hence $g$ belongs to $\calN-\calC$.    The condition $g\in \calN-\calC$ implies that $g^2$ is a scalar matrix and hence  $\bbar\rho_\Lambda^\proj(\calI_\ell)$ is a group of order $1$ or $2$.   This contradicts $d>2$, so $\beta_\Lambda$ is unramified at $\ell$.
\end{proof}

Let $\chi$ be the primitive Dirichlet character that satisfies $\beta_\Lambda(\Frob_p)=\chi(p)$ for all primes $p\nmid N\ell$.   Since $\beta_\Lambda$ is a quadratic character, Lemma~\ref{L:Cartan character} implies that the conductor of $\chi$ divides $\calM$.  The character $\chi$ is non-trivial since $\beta_\Lambda$ is non-trivial.  Assumption (\ref{P:criteria b}) implies that there is a prime $p\nmid N\ell$ satisfying $\chi(p)=-1$ and $r_p \not\equiv 0 \pmod{\lambda}$.  We thus have $g\in\calN-\calC$ and $\tr(g)\neq 0$, where $g:=\bbar\rho_{\Lambda}(\Frob_p) \in \calN$.   However, this contradicts that $\tr(A)=0$ for all $A\in \calN-\calC$.   

Therefore, the image of $\bbar\rho_\Lambda$ does not lie in the normalizer of a Cartan subgroup of $\GL_2(\FF_\Lambda)$.

\subsection{$\mathfrak{A}_5$ case}  \label{SS:A5 case}

Assume that $\bbar\rho^\proj_\Lambda(G)$  is isomorphic to $\mathfrak{A}_5$ with $\#k_\lambda \notin \{4,5\}$.

The image of $r_p/p^{k-1}= a_p^2/(\varepsilon(p)p^{k-1})$ in $\FF_\lambda$ is equal to $\tr(A)^2/\det(A)$ with $A=\bbar\rho_\Lambda(\Frob_p)$.   Every element of $\mathfrak{A}_5$ has order $1$, $2$, $3$ or $5$, so Lemma~\ref{L:easy image PGL} implies that the image of $r_p/p^{k-1}$ in $\FF_\lambda$ is  $0$, $1$, $4$ or is a root of $x^2-3x+1$ for all $p\nmid N\ell$.  If $\lambda | 5$, then Lemma~\ref{L:easy image PGL} implies that $k_\lambda=\FF_5$ which is excluded by our assumption on $k_\lambda$.   So $\lambda \nmid 5$ and Lemma~\ref{L:easy image PGL} ensures that $k_\lambda$ is the splitting field of $x^2-3x+1$ over $\FF_\ell$.   So $k_\lambda$ is $\FF_\ell$ if $\ell \equiv \pm 1 \pmod{5}$ and $\FF_{\ell^2}$ if $\ell \equiv \pm 2 \pmod{5}$.    

From assumption (\ref{P:criteria c}), we find that $\ell > 5k-4$ and $\ell \nmid N$.  By Lemma~\ref{L:inertia}, the group $\bbar\rho_\Lambda^\proj(G)$ contains an element of order at least $(\ell-1)/(k-1) > ((5k-4)-1)/(k-1) = 5$.   This is a contradiction since $\mathfrak{A}_5$ has no elements with order greater than $5$.

\subsection{$\mathfrak{A}_4$ and $\mathfrak{S}_4$ cases}  \label{SS:S4 case}

Suppose that $\bbar\rho^\proj_\Lambda(G)$  is isomorphic to $\mathfrak{A}_4$ or $\mathfrak{S}_4$ with $\#k_\lambda \neq 3$.

First suppose that $\#k_\lambda \notin \{5,7\}$.   The image of $r_p/p^{k-1}= a_p^2/(\varepsilon(p)p^{k-1})$ in $\FF_\lambda$ is equal to $\tr(A)^2/\det(A)$ with $A=\bbar\rho_\Lambda(\Frob_p)$.    Since every element of $\mathfrak{S}_4$ has order at most $4$, Lemma~\ref{L:easy image PGL} implies that $r_p/p^{k-1}$ is congruent to $0$, $1$, $2$ or $4$ modulo $\lambda$ for all primes $p\nmid N\ell$.  In particular, $k_\lambda=\FF_\ell$.    By assumption~(\ref{P:criteria d}), we must have $\ell>4k-3$ and $\ell\nmid N$.  By Lemma~\ref{L:inertia}, the group $\bbar\rho_\Lambda(G)$ contains an element of order at least $(\ell-1)/(k-1) > ((4k-3)-1)/(k-1) = 4$.   This is a contradiction since $\mathfrak{S}_4$ has no elements with order greater than $4$.

Now suppose that $\#k_\lambda\in \{5,7\}$.  By assumption (\ref{P:criteria e}), with any $\chi$, there is a prime $p\nmid N\ell$ such that $a_p^2/(\varepsilon(p)p^{k-1}) \equiv 2 \pmod{\lambda}$.    The element $g:=\bbar\rho_\Lambda^\proj(\Frob_p)$ has order $1$, $2$, $3$ or $4$.   By Lemma~\ref{L:easy image PGL}, we deduce that $g$ has order $4$.   Since $\mathfrak{A}_4$ has no elements of order $4$, we deduce that $H:=\bbar\rho^\proj_\Lambda(G)$ is isomorphic to $\mathfrak{S}_4$.  Let $H'$ be the unique index $2$ subgroup of $H$; it is isomorphic to $\mathfrak{A}_4$.   Define the character 
\[
\beta \colon G \xrightarrow{\bbar\rho_\Lambda^\proj} H \to H/H' \cong \{\pm 1\}.
\]
The quadratic character $\beta$ corresponds to a Dirichlet character  $\chi$ whose conductor divides $4^e \ell N$.  By assumption (\ref{P:criteria e}), there is a prime $p\nmid N\ell$ such that $\chi(p)=1$ and $a_p^2/(\varepsilon(p)p^{k-1})\equiv 2 \pmod{\lambda}$.    Since $\beta(\Frob_p)=\chi(p)=1$, we have $\bbar\rho_\Lambda^\proj(\Frob_p) \in H'$.  Since $H'\cong \mathfrak{A}_4$, Lemma~\ref{L:easy image PGL} implies that the image of $a_p^2/(\varepsilon(p)p^{k-1})$ in $\FF_\lambda$ is $0$, $1$ or $4$.   This contradicts $a_p^2/(\varepsilon(p)p^{k-1})\equiv 2 \pmod{\lambda}$.   

Therefore, the image of $\bbar\rho_\Lambda^\proj$ is not isomorphic to either of the groups $\mathfrak{A}_4$ or $\mathfrak{S}_4$.

\subsection{End of proof}

In \S\ref{SS:Borel case 0}, we saw that $\bbar\rho_\Lambda(G)$ is not contained in a Borel subgroup of $\GL_2(\FF_\Lambda)$. In \S\ref{SS:Cartan case}, we saw that $\bbar\rho_\Lambda(G)$ is not contained in the normalizer of a Cartan subgroup of $\GL_2(\FF_\Lambda)$.

 In \S\ref{SS:A5 case}, we showed that if $\#k_\lambda\notin \{4,5\}$, then $\bbar\rho^\proj_\Lambda(G)$ is not isomorphic to $\mathfrak{A}_5$.  We want to exclude the cases $\#k_\lambda \in \{4,5\}$ since $\PSL_2(\FF_4)$ and $\PSL_2(\FF_5)$ are both isomorphic to $\mathfrak{A}_5$.

 In \S\ref{SS:S4 case}, we showed that if $\#k_\lambda\neq 3$, then $\bbar\rho^\proj_\Lambda(G)$ is not isomorphic to $\mathfrak{A}_4$ and not isomorphic to $\mathfrak{S}_4$.  We want to exclude the case $\#k_\lambda=3$ since $\PSL_2(\FF_3)$ and $\PGL_2(\FF_3)$ are isomorphic to $\mathfrak{A}_4$ and $\mathfrak{S}_4$, respectively.

By Lemma~\ref{L:subgroups}, the group $\bbar\rho_\Lambda^{\proj}(G)$ must be conjugate in $\PGL_2(\FF_\Lambda)$ to $\PSL_2(\FF')$ or $\PGL_2(\FF')$, where $\FF'$ is a subfield of $\FF_\Lambda$.  One can then show that $\FF'$ is the subfield of $\FF_\Lambda$ generated by the set $\{ \tr(A)^2/\det(A) : A \in \bbar\rho_\Lambda(G)\}$.   By the Chebotarev density theorem, the field $\FF'$ is the subfield of $\FF_\Lambda$ generated by the images of $a_p^2/(\varepsilon(p) p^{k-1})=r_p/p^{k-1}$ with $p\nmid N\ell$.   Therefore, $\FF'=k_\lambda$ and hence $\bbar\rho_\Lambda^{\proj}(G)$ is conjugate in $\PGL_2(\FF_\Lambda)$ to $\PSL_2(k_\lambda)$ or $\PGL_2(k_\lambda)$.

\section{Examples}

\subsection{Example from \S\ref{SS:weight 3 level 27}} \label{SS:3 27}

Let $f$ be the newform from \S\ref{SS:weight 3 level 27}.  We have $E=\QQ(i)$.    We know that $\Gamma\neq 1$ since $\varepsilon$ is non-trivial.   Therefore, $\Gamma=\Gal(\QQ(i)/\QQ)$ and $K=E^{\Gamma}$ equals $\QQ$.  So $\Gamma$ is generated by complex conjugation and we have $\bbar{a}_p = \varepsilon(p)^{-1} a_p$ for $p\nmid N$.     As noted in \S\ref{SS:weight 3 level 27}, this implies that $r_p$ is a square in $\ZZ$ for all $p\nmid N$ and hence $L$ equals $K=\QQ$.   Fix a prime $\ell=\lambda$ and a prime ideal $\Lambda | \ell$ of $\OO=\ZZ[i]$.  

Set $q_1=109$ and $q_2=379$; they are primes that are congruent to $1$ modulo $27$.  Set $p_1=5$, we have $\chi(p_1)=-1$, where $\chi$ is the unique non-trivial character $(\ZZ/3\ZZ)^\times \to \{\pm 1\}$.  Set $q=5$; the field $\QQ(r_q)$ equals $K=\QQ$ and hence $\ZZ[r_{q}]=\ZZ$.

One can verify that $a_{109}=164$, $a_{379}=704$ and $a_5=-3i$, so $r_{109}=164^2$, $r_{379}=704^2$ and $r_5=3^2$.     We have 
\begin{equation} \label{E:disc fact}
r_{109} - (1+109^2)^2 = -2^2\cdot 3^3\cdot 7\cdot 19\cdot 31\cdot 317 \quad \text{ and }\quad 
r_{379} - (1+379^2)^2 = -2^2 \cdot 3^3\cdot 2647 \cdot 72173.
\end{equation} 
So if $\ell\geq 11$, then there is an $i\in \{1,2\}$ such that $\ell \neq q_i$ and $r_{q_i}\not\equiv (1+q_i^2)^2\pmod{\ell}$.  

Let $S$ be the set from \S\ref{S:effective Ribet} with the above choice of $q_1$, $q_2$, $p_1$ and $q$.  We find that $S = \{2,3,5,7,11\}$.  Theorem~\ref{T:effective version} implies that $\bbar\rho_\Lambda^\proj(G)$ is conjugate in $\PGL_2(\FF_\Lambda)$ to $\PSL_2(\FF_\ell)$ when $\ell>11$.\\

Now take $\ell \in \{7,11\}$.  Choose a prime ideal $\Lambda$ of $\OO$ dividing $\ell$.  We have $e_0=e_1=e_2=0$ and $\calM=3$.   The subfield $k_\ell$ generated over $\FF_\ell$ by the images of $r_p$ modulo $\ell$ with $p\nmid N\ell$ is of course $\FF_\ell$ (since the $r_p$ are rational integers).   We now verify the conditions of Theorem~\ref{T:criteria}.

We first check condition (\ref{P:criteria a}).   Suppose there is a character $\chi\colon (\ZZ/27\ZZ)^\times \to \FF_\ell^\times$ such that $\chi(q_2)$ is a root of $x^2-a_{q_2} x + \varepsilon(q_2) q_2^2$ modulo $\ell$.    Since $q_2\equiv 1\pmod{27}$ and $a_{q_2}=704$,  we find that $1$ is a root of $x^2-a_{q_2} x +  q_2^2 \in \FF_\ell[x]$.   Therefore, $a_{q_2} \equiv 1+q_2^2 \pmod{\ell}$ and hence $r_{q_2}^2 = a_{q_2}^2 \equiv  (1+ q_2^2)^2 \pmod{\ell}$.   Since $\ell \in\{7,11\}$, this contradicts (\ref{E:disc fact}).   This proves that condition (\ref{P:criteria a}) holds.  

We now check condition (\ref{P:criteria b}).  Let $\chi\colon (\ZZ/3\ZZ)^\times \to \{\pm 1\}$ be the non-trivial character.    We have $\chi(5)=-1$ and $r_5 = 9 \not\equiv 0 \pmod{\ell}$.   Therefore, (\ref{P:criteria b}) holds.

We now check condition (\ref{P:criteria c}).   If $\ell=7$, we have $\ell \equiv 2 \pmod{5}$ and $\#k_\ell = \ell \neq \ell^2$, so condition (\ref{P:criteria c}) holds.   Take $\ell=11$.   We have $a_5^2/(\varepsilon(5) 5^2) = 9/5^2 \equiv 3 \pmod{11}$, which verifies (\ref{P:criteria c}).

Condition (\ref{P:criteria d}) holds since $\#k_\ell = 5$ if $\ell = 7$, and $\ell>4k-3=9$ and $\ell\nmid N$ if $\ell=11$.

Finally we explain why condition (\ref{P:criteria e}) holds when $\ell=7$.   Let $\chi\colon (\ZZ/7\cdot 27 \ZZ)^\times \to\{\pm 1\}$ be any non-trivial character.    A quick computation shows that there is a prime $p\in \{13,37,41\}$ such that $\chi(p)=1$ and that $a_p^2/(\varepsilon(p) p^2) \equiv 2 \pmod{7}$.

Theorem~\ref{T:criteria} implies that $\bbar\rho_\Lambda^\proj(G)$ is conjugate in $\PGL_2(\FF_\Lambda)$ to $\PSL_2(\FF_\ell)$ or $\PGL_2(\FF_\ell)$.   Since $L=K$, the group $\bbar\rho_\Lambda^\proj(G)$ isomorphic to $\PSL_2(\FF_{\ell})$ by Theorem~\ref{T:PSL2 vs GL2}(\ref{T:PSL2 vs GL2 i}).

\subsection{Example from \S\ref{SS:weight 3 level 160}} \label{SS:3 160}
Let $f$ be a newform as in \S\ref{SS:weight 3 level 160}; we have $k=3$ and $N=160$.     The \texttt{Magma} code below verifies that $f$ is uniquely determined up to replacing by a quadratic twist and then a Galois conjugate.  So the group $\bbar\rho_\Lambda^\proj(G)$, up to isomorphism, does not depend on the choice of $f$.

{\small
\begin{verbatimtab}
    eps:=[c: c in Elements(DirichletGroup(160)) | Order(c) eq 2 and Conductor(c) eq 20][1];
    M:=ModularSymbols(eps,3);
    F:=NewformDecomposition(NewSubspace(CuspidalSubspace(M)));
    assert #F eq 2; 
    _,chi:=IsTwist(F[1],F[2],5);  
    assert Order(chi) eq 2;
\end{verbatimtab}
}

Define $b = \zeta_{13}^1 + \zeta_{13}^5+\zeta_{13}^8 +\zeta_{13}^{12}$, where $\zeta_{13}$ is a primitive $13$-th root of unity in $\Qbar$ (note that $\{1,5,8,12\}$ is the unique index $3$ subgroup of $\FF_{13}^\times$).  The characteristic polynomial of $b$ is $x^3 + x^2 - 4x + 1$ and hence $\QQ(b)$ is the unique cubic extension of $\QQ$ in $\QQ(\zeta_{13})$.  The code below shows that the coefficient field $E$ is equal to $\QQ(b,i)$ (it is a degree $6$ extension of $\QQ$ that contains roots of $x^3 + x^2 - 4x + 1$ and $x^2+1$).
{\small
\begin{verbatimtab}
    f:=qEigenform(F[1],2001);
    a:=[Coefficient(f,n): n in [1..2000]]; 
    E:=AbsoluteField(Parent(a[1]));  
    Pol<x>:=PolynomialRing(E);
    assert Degree(E) eq 6 and HasRoot(x^3+x^2-4*x+1) and HasRoot(x^2+1);
\end{verbatimtab}
}

Fix notation as in \S\ref{SS:inner twists}.  We have $\Gamma\neq 1$ since $\varepsilon$ is non-trivial.   The character $\chi_\gamma^2$ is trivial for $\gamma \in \Gamma$ (since $\chi_\gamma$ is always a quadratic character times some power of $\varepsilon$).   Therefore, $\Gamma$ is a $2$-group.     The field $K=E^\Gamma$ is thus $\QQ(b)$ which is the unique cubic extension of $\QQ$ in $E$.   Therefore, $r_p = a_p^2/\varepsilon(p)$ lies in $K=\QQ(b)$ for all $p\nmid N$.

The code below verifies that $r_3$, $r_7$ and $r_{11}$ are squares in $K$ that do not lie in $\QQ$ (and in particular, are non-zero).    Since $3$, $7$ and $11$ generate the group $(\ZZ/40\ZZ)^\times$, we deduce from Lemma~\ref{L:L description} that the field $L=K(\{\sqrt{r_p}:p\nmid N\})$ is equal to $K$. 

{\small
\begin{verbatimtab}[7]
	_,b:=HasRoot(x^3+x^2-4*x+1);  K:=sub<E|b>;
	r3:=K!(a[3]^2/eps(3)); r7:=K!(a[7]^2/eps(7)); r11:=K!(a[11]^2/eps(11));
	assert IsSquare(r3) and IsSquare(r7) and IsSquare(r11);
	assert r3 notin Rationals() and r7 notin Rationals() and r11 notin Rationals();
\end{verbatimtab}
}

\noindent Let $N_{K/\QQ}\colon K\to \QQ$ be the norm map.  The following code verifies that $N_{K/\QQ}(r_3) = 2^6$, $N_{K/\QQ}(r_7) = 2^6$, $N_{K/\QQ}(r_{11}) = 2^{12}5^4$, $N_{K/\QQ}(r_{13})=2^{12} 13^2$, $N_{K/\QQ}(r_{17}) = 2^{18}5^2$, and that
\begin{equation} \label{E:end gcd}
\gcd\Big( 641\cdot N_{K/\QQ}(r_{641}-(1+642^2)^2), \, 1061\cdot N_{K/\QQ}(r_{1061}-(1+1061^2)^2)\Big)  = 2^{12}.\\
\end{equation}
{\small
\begin{verbatimtab}[7]
	r13:=K!(a[13]^2/eps(13)); r17:=K!(a[17]^2/eps(17));
	assert Norm(r3) eq 2^6 and Norm(r7) eq 2^6 and Norm(r11) eq 2^12*5^4; 
	assert Norm(r13) eq 2^12*13^2 and Norm(r17) eq 2^18*5^2; 
	r641:=K!(a[641]^2/eps(641));   
	r1061:=K!(a[1061]^2/eps(1061)); 
	n1:=Integers()!Norm(r641-(1+641^2)^2);   
	n2:=Integers()!Norm(r1061-(1+1061^2)^2); 
	assert GCD(641*n1,1061*n2) eq 2^12;
\end{verbatimtab}
}

Set $q_1=641$ and $q_2=1061$; they are primes congruent to $1$ modulo $160$.    Let $\lambda$ be a prime ideal of $R$ dividing a rational prime $\ell>3$.     From (\ref{E:end gcd}), we find that $\ell\neq q_i$ and $r_{q_i} \not\equiv(1+q_i^2)^2 \pmod{\lambda}$ for some $i\in \{1,2\}$ (otherwise $\lambda$ would divide $2$).

Set $p_1=3$, $p_2= 7$ and $p_3=11$.    For each non-trivial quadratic characters $\chi\colon (\ZZ/40 \ZZ)^\times \to \{\pm 1\}$, we have $\chi(p_i)=-1$ for some prime $i\in\{1,2,3\}$ (since $3$, $7$ and $11$ generate the group $(\ZZ/40\ZZ)^\times$).  From the computed values of $N_{K/\QQ}(r_p)$, we find that $r_{p_i} \not\equiv 0 \pmod{\lambda}$ for all $i\in \{1,2,3\}$ and all non-zero prime ideals $\lambda \nmid N$ of $R$. 

Set $q=3$.   We have noted that $r_{q}\in K-\QQ$, so $K=\QQ(r_{q})$.     The index of the order $\ZZ[r_{q}]$ in $R$ is a power of $2$ since $N_{K/\QQ}(q)$ is a power of $2$.

Let $S$ be the set from \S\ref{S:effective Ribet} with the above choice of $q_1$, $q_2$, $p_1$, $p_2$, $p_3$ and $q$.  The above computations show that $S$ consists of the prime ideals $\lambda$ of $R$ that divide a prime $\ell\leq 11$.\\

Now let $\ell$ be an odd prime congruent to  $\pm 2$, $\pm 3$, $\pm 4$ or $\pm 6$ modulo $13$.   Since $K$ is the unique cubic extension in $\QQ(\zeta_{13})$, we find that the ideal $\lambda:= \ell R$ is prime in $R$ and that $\FF_\lambda\cong\FF_{\ell^3}$.    The above computations show that $\lambda \notin S$ when $\ell \notin \{3,7,11\}$.    Theorem~\ref{T:effective version} implies that if $\ell\notin \{3,7,11\}$, then $\bbar\rho_\Lambda^\proj(G)$ is conjugate in $\PGL_2(\FF_\Lambda)$ to $\PSL_2(\FF_\lambda)$ or $\PGL_2(\FF_\lambda)$, where $\Lambda$ is a prime ideal of $\OO$ dividing $\lambda$. So if $\ell \notin \{3,7,11\}$, the group $\bbar\rho_\Lambda^\proj(G)$ isomorphic to $\PSL_2(\FF_\lambda)\cong \PSL_2(\FF_{\ell^3})$ by Theorem~\ref{T:PSL2 vs GL2}(\ref{T:PSL2 vs GL2 i}) and the equality $L=K$.\\

Now take $\lambda = \ell R$ with $\ell \in \{3,7,11\}$; it is a prime ideal.  Choose a prime ideal $\Lambda$ of $\OO$ dividing $\lambda$.  We noted above that $\ZZ[r_3]$ is an order in $R$ with index a power of $2$; the same argument shows that this also holds for the order $\ZZ[r_7]$.  Therefore, the field $k_\lambda$ generated over $\FF_\ell$ by the images of $r_p$ modulo $\lambda$ with $p\nmid N\ell$ is equal to $\FF_\lambda$.   Since $\#\FF_\lambda =\ell^3$, we find that conditions (\ref{P:criteria c}), (\ref{P:criteria d}) and (\ref{P:criteria e}) of Theorem~\ref{T:criteria} hold.

We now show that condition (\ref{P:criteria a}) of Theorem~\ref{T:criteria} holds for our fixed $\Lambda$.   We have $e_0=0$, so take any character $\chi\colon (\ZZ/N\ZZ)^\times \to \FF_\Lambda^\times$.   We claim that $\chi(q_i)\in \FF_\Lambda$ is not a root of $x^2-a_{q_i} x + \varepsilon(q_i) q_i^2$ for some $i\in \{1,2\}$.   Since $q_i\equiv 1\pmod{N}$,  the claim is equivalent to showing that $a_{q_i} \not\equiv 1 + q_i^2 \pmod{\Lambda}$ for some $i\in \{1,2\}$.    So we need to prove that $r_{q_i} \equiv (1 + q_i^2)^2 \pmod{\lambda}$ for some $i\in \{1,2\}$; this is clear since otherwise $\ell$ divides the quantity (\ref{E:end gcd}).   This completes our verification of (\ref{P:criteria a}).

We now show that condition (\ref{P:criteria b}) of Theorem~\ref{T:criteria} holds.   We have $r_p \not\equiv 0 \pmod{\lambda}$ for all primes $p\in \{3,7,11,13,17\}$; this is a consequence of $N_{K/\QQ}(r_p)\not\equiv 0 \pmod{\ell}$.  We have $\calM= 120$ if $\ell=3$ and $\calM=40$ otherwise.  Condition (\ref{P:criteria b}) holds since $(\ZZ/\calM\ZZ)^\times$ is generated by the primes $p\in \{3,7,11,13,17\}$ for which $p\nmid \calM\ell$.

Theorem~\ref{T:criteria} implies that $\bbar\rho_\Lambda^\proj(G)$ is conjugate in $\PGL_2(\FF_\Lambda)$ to $\PSL_2(\FF_\lambda)$ or $\PGL_2(\FF_\lambda)$.   Since $L=K$, the group $\bbar\rho_\Lambda^\proj(G)$ isomorphic to $\PSL_2(\FF_\lambda)\cong \PSL_2(\FF_{\ell^3})$ by Theorem~\ref{T:PSL2 vs GL2}(\ref{T:PSL2 vs GL2 i}).

%\bibliographystyle{plain}
%\bibliography{/Users/zywina/Documents/papers/bib/master}


\end{document}